\documentclass[reqno,english]{amsart}

\usepackage{amsmath,amsfonts,amssymb,graphicx,amsthm,enumerate,url}
\usepackage[noadjust]{cite}
\usepackage{stmaryrd}
\usepackage{comment,paralist,mathrsfs}
\usepackage{mathrsfs,booktabs,tabularx}
\usepackage{xifthen,xcolor,tikz,setspace}
\usetikzlibrary{decorations.pathmorphing,patterns,shapes,calc,decorations}
\usetikzlibrary{decorations.pathreplacing}
\usepackage{mathtools}
\usepackage[small]{caption}
\usepackage{subcaption}

\usepackage[colorinlistoftodos]{todonotes}
\usepackage[colorlinks=true]{hyperref}
\numberwithin{equation}{section}

\newcommand{\bin}{\operatorname{Bin}}

\renewcommand{\epsilon}{\varepsilon}

\usepackage{color}

\newtheorem{theorem}{Theorem}[section]
\newtheorem*{theorem*}{Theorem}
\newtheorem{lemma}[theorem]{Lemma}

\newtheorem{conj}[theorem]{Conjecture}
\newtheorem{claim}[theorem]{Claim}
\newtheorem{proposition}[theorem]{Proposition}

\newtheorem{observation}[theorem]{Observation}
\newtheorem*{observation*}{Observation}

\newtheorem{corollary}[theorem]{Corollary}

\newtheorem{remark}[theorem]{Remark}

\theoremstyle{definition}{
\newtheorem{example}[theorem]{Example}

\newtheorem*{definition*}{Definition}

}

\newcommand{\R}{\mathbb R}

\newcommand{\cD}{\mathcal D}
\newcommand{\cG}{\mathcal G}

\newcommand{\sC}{\mathsf C}

\newcommand{\bv}{\mathbf r}

\newcommand{\p}{{\bf p}}

\begin{document}

\title{Sharp threshold for rigidity of random graphs}
\author{Alan Lew}\address{
Einstein Institute of Mathematics\\
 Hebrew University\\ Jerusalem~91904\\ Israel.}
\email{alan.lew@mail.huji.ac.il}
\author{Eran Nevo }
\address{Einstein Institute of Mathematics\\
 Hebrew University\\ Jerusalem~91904\\ Israel.}
\email{nevo@math.huji.ac.il}
\author{Yuval Peled}
\address{Einstein Institute of Mathematics\\
 Hebrew University\\ Jerusalem~91904\\ Israel.}
\email{yuval.peled@mail.huji.ac.il}
\author{Orit E. Raz}
\address{Einstein Institute of Mathematics\\
 Hebrew University\\ Jerusalem~91904\\ Israel.}
\email{oritraz@mail.huji.ac.il}

\begin{abstract}
We consider the Erd\H{o}s-R\'enyi evolution of random graphs, where a new uniformly distributed edge is added to the graph in every step. For every fixed $d\ge 1$, we show that with high probability, the graph becomes rigid in $\R^d$ at the very moment its minimum degree becomes $d$, and it becomes globally rigid in $\R^d$ at the very moment its minimum degree becomes $d+1$.
\end{abstract}
\maketitle
\section{Introduction}\label{sec:Intro}
A \emph{$d$-dimensional framework} is a pair $(G,\p)$ consisting of a finite simple graph $G=(V,E)$ and an embedding $\p$ of its vertices in $\R^d$. A framework is called \emph{rigid} if every continuous motion of the vertices in $\R^d$ that starts at $\p$, and preserves the lengths of all the edges of $G$, does not change the distance between any two vertices of $G$. In general, determining whether a framework is rigid is a hard problem that may depend on the particular embedding $\p$. However, if the embedding is generic, i.e., if the $d|V|$ coordinates are algebraically independent over the rationals, rigidity depends only on the underlying graph, and is, in fact, equivalent to the stronger notion of {\em infinitesimal rigidity} introduced by Asimow and Roth~\cite{AR1,AR2} (see ~\cite{graver1993book} for a comprehensive exposition of these concepts).

Infinitesimal rigidity is defined using the $|E|\times d|V|$ {\em rigidity matrix} $R(G,\p)$ which represents the partial derivatives of the square-distances between adjacent vertices, with respect to a vertex motion in $\R^d$ starting at $\p$. Namely, the columns of $R(G,\p)$ are indexed by associating $d$ coordinates for each vertex of $V$, and the rows are indexed by the edge set $E$. The row vector $\bv_{uv}$ indexed by an edge $uv\in E$ is supported on the coordinates of $u$ and $v$, where it is equal to the $d$-dimensional row vectors $\p(u)-\p(v)$ and $\p(v)-\p(u)$ respectively. 
Suppose that the affine hull of $\{\p(v):v\in V\}$ is $d$-dimensional. Then, using the isometries of $\R^d$, one can construct $\binom {d+1}2$ linearly independent vectors in the right kernel of $R(G,\p)$, whence its rank is at most $d|V| - \binom{d+1}2$. 
A framework $(G, \p)$ is said to be \emph{infinitesimally rigid} if this bound is attained, i.e., if the rank of its rigidity matrix equals $d|V|-\binom{d+1}{2}$, or, equivalently, if its right kernel consists only of vectors that are derived by the isometries of $\R^d$. A graph $G$ is called \emph{rigid in $\R^d$} (or, \emph{$d$-rigid}, for short) if it is infinitesimally rigid with respect to some embedding. 
An embedding $\p$ is called {\em generic} if its coordinates are algebraically independent over $\mathbb Q$. In such a case, $\p$ maximizes the rank of $R(G,\p)$ for every graph $G$ on the vertex set $V$, hence the rank of the rigidity matrix $R(G,\p)$, and $d$-rigidity in particular, depends only on the combinatorial structure of the graph $G$.

% (See Asimow and Roth~\cite{AR1} for more details.)\footnote{Our presentation till here closely followed the one by Jordan and Tanigawa~\cite{Jordan-Tanigawa:RigitidyThreshold}.}

%A finite simple graph $G=(V,E)$ is \emph{$d$-rigid} if a generic embedding of its vertices $q:V\longrightarrow \R^d$ is rigid, i.e., every small perturbation of $q$ that preserves the length for all edges in $E$ preserves the distance between every pair of vertices in $V$.

    The study of structural rigidity in modern mathematics goes back famously to Cauchy's rigidity theorem (see e.g. ~\cite{PROOFS_BOOK}). In the last 50 years, starting from the works of Laman~\cite{Laman} (see~\cite{firstLaman} already in 1927) and Asimow and Roth~\cite{AR1, AR2}, the mathematical problem of characterizing $d$-rigid graphs has been studied extensively, see e.g. the textbook~\cite{graver1993book} and survey~\cite{Connelly:RigiditySurvey}. In addition, rigidity (especially in low dimensions $d=2,3$) is investigated in many application areas, such as network localization~\cite{1354686rdo},  combinatorial algorithms~\cite{jacobs_algorithm,lee08}, computational biology~\cite{jacobs_protein}, structural engineering, robot motion planning, and more (see e.g. the  survey~\cite{MR3699772}).

Note that $d$-rigidity is a monotone graph property, namely if $G$ is $d$-rigid and $e$ is a two-subset of $V$, then $G\cup \{e\}$ is also $d$-rigid; indeed, for any embedding $\p$, the rank of the corresponding rigidity matrix $R(G,\p)$ can only increase by adding an extra row. 
The study of threshold probabilities of monotone properties is one of the main themes in random graph theory. Consider the Erd\H{o}s-R\'enyi $\cG(n,p)$ model of $n$-vertex random graphs where each edge appears independently with probability $p$. It is known ~\cite{Bollobas-book-RandomGraphs} that for every non-trivial monotone graph property $\mathcal P$ there is a threshold probability $p_*=p_*(n)$ such that asymptotically almost surely (a.a.s.) --- with probability tending to $1$ as $n\longrightarrow\infty$ --- $\cG(n,p)$ does not have property $\mathcal P$ if $p/p_* \to 0$, but it does have the property if $p/p_* \to \infty$. The threshold probability is called {\em sharp} if for every $\varepsilon>0$, a.a.s. $\cG(n,(1-\varepsilon)p_*)$ does not have property $\mathcal P$, but $\cG(n,(1+\varepsilon)p_*)$ does have it.

The following natural questions arise:  What is the threshold probability for $d$-rigidity of $\cG(n,p)$? Is there a sharp threshold probability?
For $d=1$, the notion of graph rigidity turns out to coincide with that of graph connectivity. In $\cG(n,p)$, the latter was studied in a seminal paper of Erd\H{o}s and Re\'nyi~\cite{MR120167}, who determined a sharp threshold probability of $p=\log n/n$. Bollob\'as and Thomason~\cite{BT85} refined this to a {\em hitting-time} result and proved that in the evolution of random graphs (formally defined below), connectivity occurs exactly when the graph contains no isolated vertices.

It is a well known (easy) fact that, for every $d\ge 1$, a graph $G$  must have minimum degree at least $d$ to be $d$-rigid. %The threshold for minimum degree $d$ is known~\cite{Bollobas-book-RandomGraphs}. 
%and given in the following theorem. 
%\begin{theorem}[{\cite{Bollobas-book-RandomGraphs}}]\label{thm:mindegthr}
%Let $\omega(n)\longrightarrow\infty$ as $n\to infty$. Then\\
%(i) if $np<\log n + (d-1) \log\log n - \omega(n)$, then a.a.s. $G\in \cG(n,p)$ has
%minimum degree at most $d-1$, and\\ 
%(ii) if $np>\log n + (d-1) \log\log n + \omega(n)$, then a.a.s. $G\in \cG(n,p)$ has
%minimum degree at least $d$. 
%\end{theorem}
Inspired by the 1-dimensional case, it is plausible to guess that the threshold probability for $d$-rigidity of $\cG(n,p)$ coincides with the known threshold of $p=(\log n+(d-1)\log\log n)/n$ for having minimum degree $d$  (see ~\cite{Bollobas-book-RandomGraphs}). 
For $d=2$, Jackson, Servatius and Servatius~\cite[Thm.4.4]{JSS-planeThreshold} proved that if $np>{\log n + \log\log n + \omega(n)}$, where $\omega(n)\longrightarrow\infty$ as $n\to \infty$, then a.a.s. $G\in\cG(n,p)$ is $2$-rigid\footnote{There appears a typo in~\cite{JSS-planeThreshold}, writing the weaker constant $2$ instead of $1$ in front of the $\log\log n$ term.}, supporting this ``guess".  For their proof they use the fact that a 6-connected graph is rigid in $\R^2$, a fact which is a consequence of a characterization of Lov\'{a}sz and Yemini~\cite{Lovasz-Yenimi} for 2-rigidity, which has no analogues in higher dimensions. For $d>2$, it was proved by Kir\'{a}ly, Theran, and Tomioka~\cite[Thm.4]{Kiraly-Theran:RigitidyThreshold} that there exists a constant $c_d>0$ such that if $np>{c_d\log n}$ then a.a.s. $G\in\cG(n,p)$ is $d$-rigid. The estimate for the constant $c_d$ was improved by Jord\'an and Tanigawa~\cite[Thm.8.5]{Jordan-Tanigawa:RigitidyThreshold}, yet yielded $c_d > d$. 
%(after correcting a typo and establishing their conjecture that the constant $s_d$ indeed equals $1$, see~\cite{ourStiffness} for details).
Thus, even the question whether $d$-rigidity has a \emph{sharp} threshold (at $p={\log n}/{n}$) was left open for all $d>2$.

In this paper we prove that the intuition from the case $d=1$ indeed extends to higher dimensions in the strongest sense. That is, for every $d\ge 1$, in the Erd\H{o}s-R\'enyi evolution of random graphs, $d$-rigidity occurs exactly when the graph ceases to contain a vertex of degree less than $d$. 

%We prove a more the hitting time result claimed in the abstract for all $d$, described next, thus improving also the threshold result for $d=2$.   
%\section{Theorem}
Formally, let $\cG(n)=\{G(n,M) ~:~0\le M\le \binom n2\}$, where $G(n,0)$ is the empty $n$-vertex graph, and $G(n,M) = G(n,M-1)\cup\{e\}$ for an edge $e$ that is sampled uniformly from the edges in the complement of $G(n,M-1)$. Fix $d\ge 1$, and consider the random variables 
\begin{align*}
M_{d-\rm rigid}&:=\min\{M: G(n,M)\text{ is $d$-rigid}\}\quad\text{  and}\\
M_d&:=\min\{M:\delta(G(n,M))=d\},
\end{align*}
where $\delta(G)$ denotes the minimum degree of a graph $G$.
Note that $G(n,M_d-1)$ is not $d$-rigid whence $M_{d-\rm rigid}\ge M_d$ holds deterministically. The main result of this paper asserts that the converse inequality a.a.s.\, holds.
%Note that 
%Theorem~\ref{thm:mindegthr} implies that
%for $pn< \log n + (d-1) \log\log n - \omega(n)$ then $M_d> p\binom{n}{2}$ a.a.s. and that for $p>\frac{\log n + (d-1) \log\log n - \omega(n)}{n}$, then $M_d\le p\binom{n}{2}$ a.a.s.
%The following is the main result of this paper.
\begin{theorem}\label{thm:1} For every $d\geq 1$, a.a.s.\, $M_{d-\rm rigid}= M_d$.
%$\mathbb P(G(n,M_d)~\mbox{is rigid in}~\R^d)\longrightarrow 1$ as $n\longrightarrow\infty.$
\end{theorem}

The proof is inspired by the argument in ~\cite{LuzP} on the vanishing of the integral homology of random $2$-complexes and is given in Sections~\ref{sec:Pr}, \ref{sec:2.2-Giant} and \ref{sec:2.3-Expansion}.

%\begin{remark}\label{rem:afterthm}
%Note that $G(n,M_d-1)$ is not $d$-rigid whence $M_{d-\rm rig}\ge M_d$ holds deterministically. Theorem~\ref{thm:1} asserts that, in addition, $M_{d-\rm rig}\le M_d$ holds a.a.s. \end{remark}
%Combined with Theorem~\ref{thm:1} this implies that $M_{rig}$

Using the contiguity of the random graph models $\cG(n,p)$ and $G(n,M)$, as well as known accurate estimates on the minimum degree of $\cG(n,p)$ (see \cite{Bollobas-book-RandomGraphs}), Theorem~\ref{thm:1} readily implies that the sharp threshold probability for $d$-rigidity of $\cG(n,p)$ coincides with the sharp threshold probability 
%of $$p=\frac{\log(n)+(d-1)\log\log(n)}n$$ 
for having minimum degree $\delta(\cG(n,p))=d$:

\begin{corollary}
For every $d\ge 1$ and function $\omega(n)\longrightarrow\infty$ as $n\to \infty$:

(i) if $np>{\log n + (d-1) \log\log n + \omega(n)}$ then  $G\in\cG(n,p)$ is a.a.s. $d$-rigid, and

(ii) if $np<{\log n + (d-1) \log\log n -\omega(n)}$ then $G\in\cG(n,p)$ is a.a.s. not $d$-rigid.

Moreover, for every $c\in\R$, if $np=\log n +(d-1)\log\log n+c$, the probability that $G\in\cG(n,p)$ is $d$-rigid tends to $e^{-e^{-c}/(d-1)!}$ as $n\longrightarrow\infty$.
\end{corollary}

\begin{remark}
In fact Theorem~\ref{thm:1} holds for every \emph{abstract rigidity matroid}, introduced by  Graver~\cite{Graver:AbstractRigidity}, with exactly the same proof. Hence, also the same sharp threshold probability  holds in that generality. 
Indeed, we only use the following facts, that hold in every abstract $d$-rigidity matroid: 
\begin{enumerate}[(i)]
\item For every rigid graph $G=(V,E)$ with $|V|\ge d+1$ there holds:
    \begin{enumerate}
        \item $\delta(G)\geq d$, and 
        \item the rank of $G$ in the matroid equals $d|V|-\binom{d+1}{2}$.
    \end{enumerate}
    \item The complete graph on $d+2$ vertices minus an edge, $K_{d+2}^-$, is rigid.
    
\end{enumerate}
\end{remark}

\subsection*{Generic Global Rigidity} A $d$-dimensional framework $(G,\p)$ is called {\em globally} $d$-rigid if every embedding of $V$ in $\R^d$ that realizes the pairwise distances $\|\p(v)-\p(u)\|$ for every edge $uv\in E$ is obtained from $\p$ by an isometry of $\R^d$. We say that a graph $G$ is globally $d$-rigid if $(G,\p)$ is globally $d$-rigid for some generic embedding $\p$. Fundamental results in the theory of global rigidity due to Connelly ~\cite{Connelly:global} and and Gortler, Healy and Thurston ~\cite{GHT}, assert that global $d$-rigidity of a graph $G$ is equivalent to the property that $(G,\p)$ is globally $d$-rigid for every generic $\p$. We refer the reader to ~\cite{tanigawa2015sufficient, jordan2017global} and the references therein for a detailed introduction to global rigidity, the differences from rigidity and infinitesimal rigidity, as well as a review of the research highlights in this topic in recent decades. 

It turns out that global $d$-rigidity is a stronger property than $d$-rigidity. For instance, the minimum degree of a globally $d$-rigid graph is at least $d+1$ (and not $d$), since it is possible to reflect a vertex of degree $d$ or less over an affine hyperplane spanned by its neighbors --- yielding a different embedding of the vertices which realizes the same distances between the adjacent pairs (see \cite{hendrickson1992conditions}). In addition, it is easy to see that global $d$-rigidity is a monotone property. In the $d=1$ case, global $1$-rigidity is equivalent to $2$-vertex-connectivity ~\cite{jordan2017global}. It is known that in the evolution of random graphs, a.a.s.\, $G(n,M)$ becomes $2$-vertex-connected at the very moment the last vertex of degree $1$ disappears~\cite{Bollobas-book-RandomGraphs}.

Building on a sufficient condition for global rigidity discovered by Tanigawa~\cite{tanigawa2015sufficient}, Jord\'an proved that every $(d+1)$-rigid graph is globally $d$-rigid ~\cite{jordan2017Aglobal}. By letting \[
M_{d-\rm GR}:=\min\{M: G(n,M)\text{ is globally $d$-rigid}\},
\]
we have that $
M_{(d+1)-\rm rigid}\ge M_{d-\rm GR}\ge M_{d+1}$ holds deterministically. Therefore, applying Theorem \ref{thm:1} for $(d+1)$-rigidity, yields, for every fixed $d\ge 1$, a hitting-time result for global $d$-rigidity: $G(n,M)$ becomes globally $d$-rigid at the very moment its minimum degree becomes $d+1$. 
 
\begin{corollary}\label{thm:global}
    For every $d\ge 1$, a.a.s.\, $M_{d-\rm GR}= M_{d+1}$.
\end{corollary}

Outline: In Section~\ref{sec:Pr} we prove Theorem~\ref{thm:1}, based on two propositions; the first shows that long before $M_d$ edges are inserted, the closure of the graph (in the rigidity matroid) already contains a giant clique, see Section~\ref{sec:2.2-Giant}, and the second shows an expansion property when $M_d$ edges are inserted, see Section~\ref{sec:2.3-Expansion}.
%In Section~\ref{sec:global} we prove Theorem \ref{thm:global} regarding global $d$-rigidity. 
We end in Section~\ref{sec:Open} with open problems concerning the threshold for the appearance of a giant $d$-rigid component. 

\section{Proof of the main Theorem~\ref{thm:1}}\label{sec:Pr}

Let $n\ge d+1$, $G=([n], E)$ and $\p:[n]\to \R^d$
%$v_1,...,v_n\in\R^d$ 
be a generic embedding of its vertices.  % of a graph $G=([n],E)$.
%, namely all $dn$ entries are algebraically independent over the rationals, 
For every $1\le x<y\le n$ consider the vector $\bv_{xy}$, defined as in Section \ref{sec:Intro}
%denote by $\bv_{xy}\in \R^{dn}$ 
(namely, this is the row vector of the edge $xy\in\binom{[n]}{2}$ in the rigidity matrix of the complete graph $K_n$).
%, namely, the restriction of $\bv_{xy}$ to the $x$-coordinates equals $v_x-v_y$, to the $y$-coordinates equals $v_y-v_x$, and to the other coordinates equals zero. 
%Given an $n$-vertex graph $G$, we d
We define the \emph{$d$-rigidity closure} of $G$ by 
\[
\sC_d(G) = \left\{ xy \in\binom{[n]}{2} ~:~\bv_{xy} \in \mbox{span}_{\R}(\bv_e~:~e\in G)\right\}. 
\]
In other words, $xy\in\sC_d(G)$ if and only if $\mbox{rank}(R(G,\p))=\mbox{rank}(R(G\cup\{xy\},\p)),$ hence $\sC_d(G)$ does not depend on the embedding $\p$ provided it is generic.
In addition, recall that $G$ is $d$-rigid if the rank of the corresponding rigidity matrix $R(G,\p)$ attains the maximal value of $dn-\binom{d+1}2$. This is equivalent, by standard linear algebra and the fact that $K_n$ is $d$-rigid, to the property that the rank of the rigidity matrix does not increase by the addition of any edge to the graph. Therefore, a graph $G$ is $d$-rigid if and only if $\sC_d(G)$ is the complete graph. 

In particular, as $K_{d+2}^-$ is $d$-rigid (since $K_{d+2}$ is $d$-rigid and the number of its edges exceeds the rank of the rigidity matrix by $1$), we have that $\sC_d(K_{d+2}^{-})=K_{d+2}$, which yields the following useful observation.
\begin{observation}\label{fact:K_minus}
Suppose that $\sC_d(G)$ contains a $d$-clique on the vertex set $\{u_1,...,u_d\}$ and in addition there exists a pair of distinct vertices $x,y\not\in\{u_1,\ldots,u_d\}$, such that each $u_i$ is adjacent to both $x$ and $y$. Then $xy$ is an edge in $\sC_d(G)$.
\end{observation}
%\begin{proof}
%Due to the fact that the graph $K_{d+2}^{-}$ is $d$-rigid, the vector $\bv_{xy}$ is a linear combination of the $\binom {d+2}2-1$ vectors $\bv_{xu_i},\bv_{yu_i},\bv_{u_iu_j}$, $1\le i\ne j \le d$, which are spanned by $\{\bv_e~:~e\in G\}$.
%\end{proof}

We derive Theorem \ref{thm:1} from the two following propositions. The first one, Proposition~\ref{lem:component}, establishes the existence of a giant clique in $\sC_d(G(n,M))$ for a sufficiently large $M$ that is much smaller than $M_d$. The second, Proposition~\ref{lem:expansion}, describes a structural feature of $G(n,M_d)$ that yields a bootstrap argument: If $\sC_d(G(n,M_d))$ contains a giant clique then the entire graph $\sC_d(G(n,M_d))$ is a clique whence $G(n,M_d)$ is $d$-rigid.

\begin{proposition}(Giant clique in the closure)\label{lem:component}
The graph $\sC_d(G(n,10d^2n))$ a.a.s. contains a clique of at least $5n/9$ vertices.
\end{proposition}

\begin{proposition}(Expansion)\label{lem:expansion}
The graph $G(n,M_d)$ a.a.s. has the property that every subset $B$ of size $1\le |B| \le n/2$ contains a vertex $v$ with at least $d$ neighbors in $[n]\setminus B$.
\end{proposition}

We are now ready to prove Theorem~\ref{thm:1}, given Propositions~\ref{lem:component} and \ref{lem:expansion}; the latter are proved in Sections~\ref{sec:2.2-Giant} and~\ref{sec:2.3-Expansion} resp.
\begin{proof}[Proof of Theorem \ref{thm:1}]
We need to prove that $M_{d-\rm rigid}\le M_d$, i.e., that $G(n,M_d)$ is a.a.s.\, $d$-rigid.
We condition on the event in Proposition~\ref{lem:component}, the event in Proposition~\ref{lem:expansion}, and the event that $M_d\ge 10d^2n$, that all occur a.a.s.\footnote{In fact, it is known that a.a.s. $M_d = (1/2+o(1))n\log n$; see \cite{Bollobas-book-RandomGraphs}.}
%Using Proposition~\ref{lem:component}, the assumption that $M_d\ge 10d^2n$ and the monotonicity of the closure, 
Using the monotonicity of the closure, we derive that 
the largest clique $A$ in $\sC_d(G(n,M_d))$ is of size $|A|\ge 5n/9$.

We claim that $A=[n]$.
%, hence $\sC_d(G(n,M_d))$ is a clique, thereby $G(n,M_d)$ is $d$-rigid, as needed. 
Indeed, otherwise, by Proposition~\ref{lem:expansion}, there is a vertex $v \in B:=[n]\setminus A$ with at least $d$ neighbors $u_1,...,u_d$ in $A$. In such a case, for every $x\in A \setminus \{u_1,...,u_d\}$, the vertices $u_1,...,u_d$ form a clique that is connected to both $x$ and $v$ (in the closure), hence $vx\in\sC_d(G(n,M_d))$ by Observation \ref{fact:K_minus}. Consequently, $v$ is connected to all the vertices in $A$, hence $A\cup\{v\}$ induces a clique in $\sC_d(G(n,M_d))$ --- contradicting the maximality of $A$. Thus $B$ must be the empty set.
We conclude that $A=[n]$, hence $\sC_d(G(n,M_d))$ is the complete graph, thereby $G(n,M_d)$ is $d$-rigid, as needed. 
\end{proof}

\section{Proof of Proposition~\ref{lem:component}}\label{sec:2.2-Giant}
%\subsection{A giant clique in the closure}
For the proof we need the following lemma.
\begin{lemma}\label{lem:closure}
Fix real numbers $\delta,c$ such that  $0<\delta<1$ and $c\delta>d$. Then a.a.s.\,,
$$|\sC_d(G(n,cn))| 
\ge \left(1-\delta\right)\binom {n}{2}\,.
$$
\end{lemma}
\begin{proof}
Let $1 \le M \le \binom n2$ be an integer. 
We  construct a coupling $\cD$ of pairs $(r,G)$ such that $r\sim \mathrm{U}([0,1]^M)$, $G\sim G(n,M)$ and for every $(r,G)\sim\cD$, at most $dn-\binom{d+1}2$ indices $1\le i\le M$ satisfy $r_i \le 1-  |\sC_d(G)|/{\binom n2}.$

We first show how to derive the lemma given such a coupling where $M=cn$. For this purpose, we simply note that the event $|\sC_d(G)| < (1-\delta)\binom n2$ is contained in the event that at most $dn-\binom{d+1}2$ of the $r_i$'s are smaller than $\delta.$ In other words,
\begin{equation}\label{eq:binom_bound}
\mathbb P\left(|\sC_d(G)| < (1-\delta)\binom n2\right) \le 
\mathbb P\left(\bin(cn,\delta) \le dn - \binom{d+1}2\right) \longrightarrow 0
\end{equation}
as $n\longrightarrow\infty$, by the law of large numbers, using $c\delta > d$. 

To construct the coupling $\cD$, we sample $r_1,...,r_M$ i.i.d.\, uniformly from $[0,1]$, and use its randomness to sample $(G(n,i))_{i=1}^{M}$ as follows. In every step $i$, the edge $e_i$ that we add to $G(n,i-1)$ to create $G(n,i)$ is sampled uniformly from 
\[
     e_i\sim \binom{[n]}{2}
     \setminus\sC_d(G(n,i-1))\,,\,\,\,\mbox{~if~}\,
    r_i < \frac{\binom n2-|\sC_d(G(n,i-1))|}{\binom n2 - (i-1)},
 \]
    and $e_i\sim\sC_d(G(n,i-1))\setminus G(n,i-1)$ otherwise.
In words, we first decide (using $r_i$) whether $e_i$ is sampled from the closure of $G(n,i-1)$ or not, and then we sample it uniformly. Note that the threshold we choose guarantees that $e_i$ is uniformly distributed in $\binom{[n]}2\setminus G(n,i-1)$. Therefore, $G:=G(n,M)$ is a uniform random $n$-vertex graph with $M$ edges. The key observation is that every time $e_i$ is chosen from outside the closure of $G(n,i-1)$ --- i.e., the condition
\[
r_i < \frac{\binom n2-|\sC_d(G(n,i-1))|}{\binom n2 - (i-1)}
\]
holds true --- the rank of the rigidity matrix of the obtained graph increases by one, and this can occur at most $dn-\binom{d+1}2$ times, which is the rank for the complete graph $K_n$. The construction is concluded by observing that 
\[
1-\frac{|\sC_d(G)|}{\binom n2} \le \frac{\binom n2-|\sC_d(G(n,i-1))|}{\binom n2 - (i-1)},
\]
using the monotonicity of the closure.
\end{proof}

\begin{proof}[Proof of Proposition~\ref{lem:component}]
Set $\delta = (9d)^{-1}$, $c=10d^2$ and $G=G(n,cn)$. By Lemma~\ref{lem:closure}, $\sC_d(G)$ has a.a.s.\,
at least $(1-(9d)^{-1})\binom n2$ edges.  

For every $v\in [n],$ let $d_v$ denote the degree of $v$ in $\sC_d(G)$. 
Write
%partition the vertex set of $G$ into two sets 
$[n] = A\cup B$, a partition of the vertex set of $G$, where $B=\{v~:~d_v \le (1-(4d)^{-1})(n-1)\}$ and $A=[n]\setminus B$.
%$v\in A$ otherwise. 
By double-counting the number of non-edges in $\sC_d(G)$ we find that
\[
|B|(4d)^{-1}(n-1) \le \sum_{v\in B}(n-1-d_v)\le 2\left(\binom n2-|\sC_d(G)|\right) \le 2(9d)^{-1}\binom n2\,.
\]
Thus $|B| \le 4n/9$, and so $|A|\ge 5n/9$.

We claim that $A$ induces a clique in $\sC_d(G)$. In fact, for every $v_1,v_2 \in A$ we find vertices $v_3,...,v_{d+2}$ in $A$ such that $v_i$ is connected to the vertices $v_1,...,v_{i-1}$ in the graph $\sC_d(G)$ for every $i=3,...,d+2$. By Observation \ref{fact:K_minus}, this implies that $v_1v_2\in\sC_d(G)$, hence $A$ is a clique as claimed. The construction of the vertices $v_i,~i=3,...,d+2$, is sequential and greedy. Namely, after the selection of $v_1,...,v_{i-1}$, a vertex in $[n]$ cannot be chosen as $v_i$ if it belongs to $B$ or if it is not adjacent to one of the previous vertices. Using the fact that $v_1,...,v_{i-1}\in A$, we find that the number of non-admissible choices for $v_i$ is at most
\[
|B|+\sum_{j=1}^{i-1}(n-d_{v_j}) \le 4n/9 + (d+1)(n/4d + 1-(4d)^{-1}) < n\,,
\]
for $n$ sufficiently large, and every $d\ge 1$. 
Hence there is at least one admissible choice, which proves our claim. This completes the proof of the proposition.
\end{proof}

\begin{remark}In Proposition~\ref{lem:component} we show that the closure of $G=G(n,cn)$ a.a.s.\, has  a giant clique if $c=10d^2$. Most reasonably, this occurs because $G$ itself contains a giant $d$-rigid component --- an inclusion-maximal vertex subset of positive density that induces a rigid subgraph --- but our proof does not guarantee that. In Example \ref{example:compVSclique} below we illustrate the difference between a giant clique in the closure and a rigid component by constructing a graph whose closure contains a giant clique $A$ despite the fact that the subgraph of $G$ induced by $A$ is empty, and all $d$-rigid components have constant size. We note that the graph that we construct for this illustration is very unlikely to appear as a subgraph of a random graph.
\end{remark}
\begin{example}\label{example:compVSclique}
Let $A$ be a finite set arbitrarily large, and $H$ a minimally $d$-rigid graph on $A$. For every edge $xy$ of $H$ let $G_{xy}$ be the graph $K_{d+2}^-$ with $xy$ its missing edge and its other vertices are not in $A$ and are unique to $G_{xy}$. Let $G=\bigcup_{xy\in E(H)}G_{xy}$. Then the closure is $\sC_d(G)=G\cup K_A$ and in particular it contains a clique of size $|A|$, and $|A|\ge |V(G)|/(d^2+1)$. In contrast, the maximal $d$-rigid components of $G$ are the $G_{xy}$'s, each has size $d+2$. 
\end{example}

\section{Proof of Proposition~\ref{lem:expansion}}\label{sec:2.3-Expansion}
% Let $0<\epsilon<1$ and $d\ge 1$ be an integer. We say that an $n$-vertex graph $G$ has property  $\cA_{\epsilon,d}$ if in every nonempty vertex subset $B$ of size $|B|\le \epsilon n$ there is a vertex $v$ with at least $d$ neighbors outside $B$. Note that $\cA_{\epsilon,d}$ is a monotone graph property, and that $\cA_{\epsilon,d} \subseteq \cA_{\epsilon',d}$ if $\epsilon>\epsilon'$. We need to show that the graph $G(n,M_d)$ a.a.s. has property $\cA_{1/2,d}$. Note that the constant $1/2$ is arbitrary and can be replaced by any other constant (using the same proof).

% We first note that since the minimum degree in $G(n,M_d)$ is $d$, no subset $B$ of size $1$ prevents the graph from having the property $\cA_{1/2,d}$.

Denote
\[
p_{\pm}=\frac{\log n +(d-1)\log\log n \pm \log\log\log n}{n}.
\]
There exists a standard coupling $(G_-,G(n,M_d),G_+)$ where $G_\pm\sim \cG(n,p_\pm)$, and the event $G_-\subset G(n,M_d)\subset G_+$ a.a.s.\, occurs. Indeed, let $(r_e~:~e\in\binom n2)$ be i.i.d.\, standard uniform random variables. To sample a $\cG(n,p)$ random graph we take all the edges $e$ such that $r_e < p$. Similarly, to sample the evolution $\{G(n,M)~:~0\le M\le\binom n2\}$ of random graphs we sort the edges in increasing order according to the values $r_e$. The coupling $(G_-,G(n,M_d),G_+)$ is obtained by using the same random sequence $(r_e)$ to sample all three graphs. In this coupling, the graph $G_-$ is always contained in $G_+$ and the $G(n,M_d)$ graph is sandwiched between them if $\delta(G_-)<d$ and $\delta(G_+)\ge d$, which a.a.s.\, occurs.

Recall that a vertex subset $B$ is called {\em independent} in a graph if no two vertices of $B$ are adjacent.

% \begin{claim}\label{clm:Gpm}
% Let $(G_-,G_+)$ be sampled as above. Then, a.a.s.\,  there is no subset $B$ of size $2\le |B| \le n/2$ such that 
% \begin{itemize}
%     \item No vertex of $B$ has at least $d$ neighbors outside of $B$ in the graph $G_-$, and 
%     \item $B$ is not independent in the graph $G_+$.
% \end{itemize}
% \end{claim}
\begin{claim}\label{clm:Gpm}
Let $(G_-,G_+)$ be sampled as above. Then, a.a.s.\,  every subset $B$ of size $1\le |B| \le n/2$ satisfies that either
\begin{itemize}
    \item There exists a vertex in $B$ with at least $d$ neighbors outside of $B$ in the graph $G_-$, or
    \item $B$ is independent in the graph $G_+$.
\end{itemize}
\end{claim}
\begin{proof}
We show that the expected number of sets $B$ that violate both these conditions tends to $0$ as $n\longrightarrow\infty$, from which the claim follows by the first-moment method.

Note that the first condition in the claim depends on the edges between $B$ and its complement, whereas the second condition depends on the edges within $B$. Therefore, the two conditions are independent. In consequence, a set $B$  of size $b$ violates both conditions with probability %$B$ violates the statement with probability
\[
\mathbb P(\bin(n-b,p_-)\le d-1)^b\cdot
\mathbb P\left(\bin\left(\binom b2,p_+\right)\ge 1\right).
\]
We bound the first probability by computing, for every $0\le j\le d-1,$
\begin{align*}
\nonumber 
\mathbb P(\bin(n&-b,p_-)=j) =     
\binom {n-b}jp_-^j(1-p_-)^{n-b-j} \\ 
& \le (np_-)^j\exp{\left( -\log n -(d-1)\log\log n + \log\log\log n + (b+j)p_- \right)} \\
& \le (1+o(1))\frac{(\log n)^{j}\log\log n}{n(\log n)^{d-1}}e^{bp_-}.
\end{align*}
In the first inequality we used $\binom{n-b}j \le n^j$ and $1-p_-\le e^{-p_-}$. The second inequality is derived by $(np_-)^j \le (1+o(1))(\log n)^j$ and $jp_{-}=o(1)$ whence $e^{jp_-}=1+o(1)$. Therefore, using $\sum_{j=0}^{d-1} (\log n)^{j}=(1+o(1))(\log n)^{d-1}$, we find that
\begin{equation}
\mathbb P(\bin(n-b,p_-)\le d-1) \le
(1+o(1))\frac{\log\log n}{n}e^{bp_-}\,.
     \label{eq:prob_bin}
\end{equation}

Using the standard bound $\binom nb\le(en/b)^b$, \eqref{eq:prob_bin} suffices to prove the claim for all sets of size $(\log\log n)^2 \le b\le  n/2$ (even without using the condition on $G_+$). Indeed,
\begin{align}
\sum_{b=(\log\log n)^2}^{n/2}&
\binom nb \mathbb P(\bin(n-b,p_-)\le d-1)^b\nonumber\\
&\le
\sum_{b=(\log\log n)^2}^{n/2}
\left(\frac{en}{b} (1+o(1))\frac{\log\log n}{n}e^{bp_-}\right)^b
\nonumber\\
\label{eq:large_b}
&=  
\sum_{b=(\log\log n)^2}^{n/2}
\left((e+o(1))\frac{\log\log n}{b}e^{bp_-}\right)^b\,.
\end{align}
Note that the continuous function $b\mapsto e^{bp_-}/b$ is convex in the interval $b\in[(\log\log n)^2,n/2]$, hence its maximum is attained in one of the endpoints of the interval. 
\begin{itemize}
    \item If $b=n/2$ then $bp_-=(1/2+o(1))\log n$ whence $e^{bp_-}/b=n^{-1/2+o(1)}$.
    \item If $b=(\log \log n)^2$ then $bp_-=o(1)$ whence a larger value of $$e^{bp_-}/b=(1+o(1))(\log\log n)^{-2}$$ is obtained.

\end{itemize}
 
In conclusion, \eqref{eq:large_b} is bounded from above by the geometric sum
\[
\sum_{b=(\log\log n)^2}^{n/2}
\left(\frac{e+o(1)}{\log\log n}\right)^b\longrightarrow 0\,,
\]
as $n\longrightarrow\infty.$

We turn to consider sets of sizes $1\le b\le (\log\log n)^2$. Note that in such a case,
\[
\mathbb P\left(\bin\left(\binom b2,p_+\right)\ge 1\right)
\le
\binom b2p_+ = n^{-1+o(1)}.
\]
Additionally, we absorb the factor $e^{bp_-}$ in \eqref{eq:prob_bin} into the $1+o(1)$ factor to find that
\begin{align}
\nonumber
\sum_{b=1}^{(\log\log n)^2}
\binom nb &
\mathbb P(\bin(n-b,p_-)\le d-1)^b\cdot
\mathbb P\left(\bin\left(\binom b2,p_+\right)\ge 1\right)
\\
\label{eq:second_case}
&\le
\sum_{b=1}^{(\log\log n)^2}
\left((e+o(1))\frac{\log\log n}{b}\right)^b\cdot n^{-1+o(1)}.
\end{align}
By applying the bound 
$
\left((e+o(1))\log\log n/b\right)^b \le (\log n)^{1+o(1)}
$
(which is obtained by analysis of this function with respect to $b$) to each of the $(\log\log n)^2$ summands, \eqref{eq:second_case} is bounded from above by
\[
(\log\log n)^2\cdot (\log n)^{1+o(1)}\cdot n^{-1 +o(1)} \longrightarrow 0\,,
\]
as $n\longrightarrow\infty$, which concludes the proof of the claim.
\end{proof}

We derive Proposition~\ref{lem:expansion} as follows. Sample $(G_-,G(n,M_d),G_+)$ as above and condition on the event that  $G_-\subset G(n,M_d) \subset G_+$ which a.a.s.\, occurs.

Suppose that every vertex in a set $B$ has less than $d$ neighbors outside of $B$ in the graph $G(n,M_d)$. 
By monotonicity, no vertex in $B$ can have more than $d-1$ neighbors outside of $B$ in the graph $G_-$.
On the other hand, since the minimum degree in $G(n,M_d)$ is $d$, there must be an edge within $B$ in $G(n,M_d)$ --- which also appears in $G_+$, hence $B$ is not independent in $G_+$. By Claim \ref{clm:Gpm}, a.a.s. no such set $B$ of size $1\le |B| \le n/2$ exists.
\qed

\section{Open Problems}\label{sec:Open}
The emergence of a giant connected component in $\cG(n,p)$ is one of the most important phenomena in random graph theory. 
It is therefore natural to study the appearance of a giant $d$-rigid component --- an inclusion-maximal $d$-rigid induced subgraph of linear size --- in a $\cG(n,p)$ random graph. Here is a conjecture extending the $d=2$ case, which was established in~\cite{Kasiviswanathan-Moore-Theran:giant-2-rigid}, and studied further in~\cite{lelarge18}. 

First, we observe two necessary conditions the induced subgraph $G_A$ must satisfy if $A\subset [n]$ with at least $d+1$ vertices is a $d$-rigid component of $G\sim G(n,p)$: 
\begin{enumerate}[(i)]
    \item The minimum degree in $G_A$ is at least $d$, and \item There are at least $d|A|-\binom{d+1}{2}$ edges in $G_A$.
\end{enumerate}
    
In brief, the conjecture below states that in $\cG(n,p)$, these necessary conditions are also sufficient. I.e., if a subset satisfying these two conditions appears in $\cG(n,p)$ it a.a.s.\, induces a giant $d$-rigid component.

In more details, note that the second condition is closely related to the property of $d$-orientability --- the existence of an orientation of the edges of $G$ such that the maximum in-degree is at most $d$ --- since, by min-cut-max-flow duality, a subset $A$ inducing more than $d|A|$ edges is the only obstacle for $d$-orientability. The problem of orientability of random graphs was considered in ~\cite{orient1,orient2}, and a sharp threshold probability of the form $p=c_d/n$, where $c_d$ is an explicit constant, was determined. In fact, it was shown to coincide with the threshold probability for the property that the average degree of the $(d+1)$-core of $\cG(n,p)$ exceeds $2d$. Recall that the $k$-core of a graph is the largest induced subgraph of minimum degree at least $k$. It is known that for $k\ge 3$, once the $k$-core of $\cG(n,p)$ emerges it contains a positive fraction of the vertex set. We refer the reader to ~\cite{orient1,orient2,PSW} for explicit descriptions of the critical constant $c_d$ and the typical density of the $(d+1)$-core. We conjecture that at the same threshold probability $p=c_d/n$, the $(d+1)$-core $C$ of $G$ becomes $d$-rigid. To describe the entire $d$-rigid component, note that if $A$ induces a $d$-rigid subgraph and $v\notin A$ has $d$ neighbors in $A$ the $A\cup\{v\}$ induces a $d$-rigid subgraph as well. This leads us to the definition of the $((d+1)+d)$-core $\hat C$ of $G$, that is obtained from the $(d+1)$-core $C$ of $G$ by adding to it, for as long as possible, a vertex with at least $d$ neighbors in $\hat C$.

\begin{conj}\label{conj:giant}
Let $d\ge 2$, $c>0$ and $G\sim \cG(n,c/n)$. Then a.a.s.\,,
\begin{itemize}
    \item If $c<c_d$ then there is no $d$-rigid component in $G$ with more than $3$ vertices.
        \item If $c>c_d$ then $G$ contains a unique $d$-rigid component of positive density, which is comprised of the vertices of its $((d+1)+d)$-core.
\end{itemize}
\end{conj}
In the subcritical regime, we believe that known methods from, e.g., ~\cite{orient1,orient2,lelarge2} can be utilized to show that no subset $A$ with more than $3$ vertices satisfy the necessary conditions (i),(ii) above. On the other hand, establishing the $d$-rigidity of the $(d+1)$-core in the supercritical regime seems to require new ideas.

One can sharpen this conjecture and propose that in the evolution of random graphs this phase transition occurs a.a.s.\, at the very moment that the $(d+1)$-core of $G$ exists, and the numbers $n',m'$ of its vertices and edges resp. satisfy $m'\ge dn'-\binom{d+1}{2}$ for the first time.

An additional extension of Conjecture \ref{conj:giant} regards the emergence of a giant global $d$-rigid component in $\cG(n,p)$. Clearly, this cannot occur before the emergence of a giant $d$-rigid component. In addition, the necessary condition for global $d$-rigidity --- of having minimum degree $d+1$ --- is satisfied by the $(d+1)$-core, which contains a positive fraction of the vertex set in the supercritical regime of the conjecture. Therefore, it is plausible to conjecture that if $p=c/n$ and $c>c_d$, then a.a.s.\, the $(d+1)$-core of $G$ constitutes a giant globally $d$-rigid component.

% Additionally, the $(k+k')$-core $H$ of $G$ is obtained by a process in which $H$ is initially the $k$-core of $G$, and vertices with at least $k'$ neighbors in $H$ are added to $H$ as long as possible.

% % For every $c>0$, let $t=t(c)$ be the largest root in $[0,\infty)$ of the equation
% % $
% % t = \mathbb P(\mbox{Poi}(ct)\ge d).
% % $
% % It is known ~\cite{PSW} that a $(d+1)$-core emerges in $\cG(n,p)$ at $p=c/n$ where $c = \inf\{c~:~t(c)>0\}$. In addition, let $X=X(c)$ be a Poisson random variable with parameter $ct(c)$. Define $c_d$ as the parameter $c$ satisfying the equation
% % \[
% % \E[X(c)~|~X(c)\ge d+1] = 2d.
% % \]
% % In words, $p=c_d/n$ is the sharp threshold probability for the property that the average degree in the $(d+1)$-core of $\cG(n,p)$ is at least $2d$.

% Note that for every $\epsilon>0$, if $p<(1-\epsilon)c_d/n$ then $G\in \cG(n,p)$ has a.a.s. no giant $d$-rigid component.

% An even bolder conjecture can be stated:
% \begin{conj}
% Let $d\ge 3.$ For every $0 \le M\le \binom n2$ let $n'$ and $m'$ denote the number of vertices and edges in the $(d+1)$-core of $G(n,M)$ respectively. Let $M_*$ be smallest $M$ such that $m'=dn'-\binom{d+1}{2}$. Then, a.a.s\,,
% \begin{enumerate}
%     \item There is no $d$-rigid component in $G(n,M_*-1)$ with more than $3$ vertices. 
%     \item $G(n,M_*)$ contains a giant $d$-rigid component of positive density, which is comprised of the vertices of the $((d+1)+d)-$core.
% \end{enumerate}
% \end{conj}

\textbf{Acknowledgements.}
We thank the anonymous referee for very helpful comments that greatly improved the presentation.
\bibliographystyle{abbrv}
\bibliography{rigid}

\begin{thebibliography}{10}

\bibitem{PROOFS_BOOK}
M.~Aigner and G.~M. Ziegler.
\newblock Proofs from the book.
\newblock {\em Berlin. Germany}, 1999.

\bibitem{AR1}
L.~Asimow and B.~Roth.
\newblock The rigidity of graphs.
\newblock {\em Trans. Amer. Math. Soc.}, 245:279--289, 1978.

\bibitem{AR2}
L.~Asimow and B.~Roth.
\newblock The rigidity of graphs. {II}.
\newblock {\em J. Math. Anal. Appl.}, 68(1):171--190, 1979.

\bibitem{lelarge18}
J.~Barr\'{e}, M.~Lelarge, and D.~Mitsche.
\newblock On rigidity, orientability, and cores of random graphs with sliders.
\newblock {\em Random Structures Algorithms}, 52(3):419--453, 2018.

\bibitem{Bollobas-book-RandomGraphs}
B.~Bollob\'{a}s.
\newblock {\em Random graphs}, volume~73 of {\em Cambridge Studies in Advanced
  Mathematics}.
\newblock Cambridge University Press, Cambridge, second edition, 2001.

\bibitem{BT85}
B.~Bollob{\'a}s and A.~Thomason.
\newblock Random graphs of small order.
\newblock In {\em North-Holland Mathematics Studies}, volume 118, pages 47--97.
  Elsevier, 1985.

\bibitem{orient2}
J.~A. Cain, P.~Sanders, and N.~Wormald.
\newblock The random graph threshold for {$k$}-orientability and a fast
  algorithm for optimal multiple-choice allocation.
\newblock In {\em Proceedings of the {E}ighteenth {A}nnual {ACM}-{SIAM}
  {S}ymposium on {D}iscrete {A}lgorithms}, pages 469--476. ACM, New York, 2007.

\bibitem{Connelly:RigiditySurvey}
R.~Connelly.
\newblock Rigidity.
\newblock In {\em Handbook of convex geometry, {V}ol.\ {A}, {B}}, pages
  223--271. North-Holland, Amsterdam, 1993.

\bibitem{Connelly:global}
R.~Connelly.
\newblock Generic global rigidity.
\newblock {\em Discrete Comput. Geom.}, 33(4):549--563, 2005.

\bibitem{MR120167}
P.~Erd\H{o}s and A.~R\'{e}nyi.
\newblock On random graphs. {I}.
\newblock {\em Publ. Math. Debrecen}, 6:290--297, 1959.

\bibitem{1354686rdo}
T.~Eren, O.~Goldenberg, W.~Whiteley, Y.~Yang, A.~Morse, B.~Anderson, and
  P.~Belhumeur.
\newblock Rigidity, computation, and randomization in network localization.
\newblock In {\em IEEE INFOCOM 2004}, volume~4, pages 2673--2684 vol.4, 2004.

\bibitem{orient1}
D.~Fernholz and V.~Ramachandran.
\newblock The {$k$}-orientability thresholds for {$G_{n,p}$}.
\newblock In {\em Proceedings of the {E}ighteenth {A}nnual {ACM}-{SIAM}
  {S}ymposium on {D}iscrete {A}lgorithms}, pages 459--468. ACM, New York, 2007.

\bibitem{GHT}
S.~J. Gortler, A.~D. Healy, and D.~P. Thurston.
\newblock Characterizing generic global rigidity.
\newblock {\em Amer. J. Math.}, 132(4):897--939, 2010.

\bibitem{graver1993book}
J.~Graver, B.~Servatius, and H.~Servatius.
\newblock {\em Combinatorial Rigidity}.
\newblock Graduate studies in mathematics. American Mathematical Society, 1993.

\bibitem{Graver:AbstractRigidity}
J.~E. Graver.
\newblock Rigidity matroids.
\newblock {\em SIAM J. Discrete Math.}, 4(3):355--368, 1991.

\bibitem{hendrickson1992conditions}
B.~Hendrickson.
\newblock Conditions for unique graph realizations.
\newblock {\em SIAM journal on computing}, 21(1):65--84, 1992.

\bibitem{JSS-planeThreshold}
B.~Jackson, B.~Servatius, and H.~Servatius.
\newblock The 2-dimensional rigidity of certain families of graphs.
\newblock {\em J. Graph Theory}, 54(2):154--166, 2007.

\bibitem{jacobs_algorithm}
D.~J. Jacobs and B.~Hendrickson.
\newblock An algorithm for two-dimensional rigidity percolation: the pebble
  game.
\newblock {\em Journal of Computational Physics}, 137(2):346--365, 1997.

\bibitem{jacobs_protein}
D.~J. Jacobs, A.~J. Rader, L.~A. Kuhn, and M.~F. Thorpe.
\newblock Protein flexibility predictions using graph theory.
\newblock {\em Proteins: Structure, Function, and Bioinformatics},
  44(2):150--165, 2001.

\bibitem{jordan2017Aglobal}
T.~Jord{\'a}n.
\newblock Extremal problems and results in combinatorial rigidity.
\newblock In {\em Proc. Hungarian Japanese Symposium on Discrete Mathematics
  and Its Applications}, pages 297--304, 2017.

\bibitem{Jordan-Tanigawa:RigitidyThreshold}
T.~Jord\'{a}n and S.~Tanigawa.
\newblock Rigidity of random subgraphs and eigenvalues of stiffness matrices.
\newblock {\em Egerv\'{a}ry Research Group, www.cs.elte.hu/egres}, TR-2020-08,
  2020.

\bibitem{jordan2017global}
T.~Jord{\'a}n and W.~Whiteley.
\newblock Global rigidity.
\newblock In {\em Handbook of Discrete and Computational Geometry}, pages
  1661--1694. Chapman and Hall/CRC, 2017.

\bibitem{Kasiviswanathan-Moore-Theran:giant-2-rigid}
S.~P. Kasiviswanathan, C.~Moore, and L.~Theran.
\newblock The rigidity transition in random graphs.
\newblock In {\em Proceedings of the {T}wenty-{S}econd {A}nnual {ACM}-{SIAM}
  {S}ymposium on {D}iscrete {A}lgorithms}, pages 1237--1252. SIAM,
  Philadelphia, PA, 2011.

\bibitem{Kiraly-Theran:RigitidyThreshold}
F.~J. Kir\'{a}ly and L.~Theran.
\newblock Coherence and sufficient sampling densities for reconstruction in
  compressed sensing.
\newblock {\em arXiv:1302.2767}, 2013.

\bibitem{Laman}
G.~Laman.
\newblock On graphs and rigidity of plane skeletal structures.
\newblock {\em Journal of Engineering Mathematics}, 4:331--340, 1970.

\bibitem{lee08}
A.~Lee and I.~Streinu.
\newblock Pebble game algorithms and sparse graphs.
\newblock {\em Discrete Mathematics}, 308(8):1425--1437, 2008.

\bibitem{lelarge2}
M.~Lelarge.
\newblock A new approach to the orientation of random hypergraphs.
\newblock In {\em Proceedings of the {T}wenty-{T}hird {A}nnual {ACM}-{SIAM}
  {S}ymposium on {D}iscrete {A}lgorithms}, pages 251--264. ACM, New York, 2012.

\bibitem{Lovasz-Yenimi}
L.~Lov\'{a}sz and Y.~Yemini.
\newblock On generic rigidity in the plane.
\newblock {\em SIAM J. Algebraic Discrete Methods}, 3(1):91--98, 1982.

\bibitem{LuzP}
T.~{\L}uczak and Y.~Peled.
\newblock Integral homology of random simplicial complexes.
\newblock {\em Discrete \& Computational Geometry}, 59(1):131--142, 2018.

\bibitem{PSW}
B.~Pittel, J.~Spencer, and N.~Wormald.
\newblock Sudden emergence of a giant k-core in a random graph.
\newblock {\em Journal of Combinatorial Theory, Series B}, 67(1):111--151,
  1996.

\bibitem{firstLaman}
H.~Pollaczek-Geiringer.
\newblock Über die gliederung ebener fachwerke.
\newblock {\em ZAMM - J. Appl. Math. Mech./Z. Angew. Math. Mech.}, 7(1):58--72,
  1927.

\bibitem{MR3699772}
J.~Sidman and A.~St.~John.
\newblock The rigidity of frameworks: theory and applications.
\newblock {\em Notices Amer. Math. Soc.}, 64(9):973--978, 2017.

\bibitem{tanigawa2015sufficient}
S.~Tanigawa.
\newblock Sufficient conditions for the global rigidity of graphs.
\newblock {\em Journal of Combinatorial Theory, Series B}, 113:123--140, 2015.

\end{thebibliography}

\end{document}